\documentclass{amsart}
\usepackage{amssymb,tikz}
\theoremstyle{plain}
\newtheorem{theorem}{Theorem}[section]
\newtheorem{claim}{Claim}[theorem]
\newtheorem*{claim*}{Claim}

\newtheorem{lemma}[theorem]{Lemma}
\newtheorem{prop}[theorem]{Proposition}

\theoremstyle{remark}

\newtheorem{rmk}[theorem]{Remark}
\newtheorem*{rmk*}{Remark}

\newtheorem*{note*}{Note}
\theoremstyle{definition}
\newtheorem{defn}[theorem]{Definition}
\newtheorem*{defn*}{Definition}

\newtheorem{example}[theorem]{Example}
\newtheorem*{example*}{Example}
\newtheorem*{examples*}{Examples}

\numberwithin{equation}{section}

 \DeclareMathOperator{\id}{id}

\DeclareMathOperator{\dom}{dom}  \DeclareMathOperator{\lsp}{span}
  
 \DeclareMathOperator{\clsp}{{\overline{span}}}

\DeclareMathOperator{\qc}{{\square_{Claim}}}

\newcommand{\eps}{\varepsilon}

\newcommand{\pc}{\renewcommand{\qedsymbol}{$\qc$}}

\newcommand{\field}[1]{\mathbb{#1}}
\newcommand{\CC}{\field{C}}

\newcommand{\NN}{\field{N}}

\newcommand{\Gg}{{\mathcal G}}

\newcommand{\Zz}{{\mathcal Z}}

\newcommand{\FE}{{\mathcal{FE}}}

\title{The Path Space of a Directed Graph}

\author[S.B.G. Webster]{Samuel B.G. Webster}
\address{Samuel Webster\\
School of Mathematics and Applied Statistics\\
University of Wollongong\\
NSW 2522\\
AUSTRALIA}
\email{sbgwebster@gmail.com}

\subjclass[2010]{Primary 46L05}
\keywords{Graph algebra, directed graph}
\thanks{This research was supported by the ARC Discovery Project DP0984360.}
\date{February 7, 2011.}

\begin{document}
\begin{abstract}
We construct a locally compact Hausdorff topology on the path space of a directed graph $E$, and identify its boundary-path space $\partial E$ as the spectrum of a commutative $C^*$-subalgebra $D_E$ of $C^*(E)$. We then show that $\partial E$ is homeomorphic to a subset of the infinite-path space of any desingularisation $F$ of $E$. Drinen and Tomforde showed that we can realise $C^*(E)$ as a full corner of $C^*(F)$, and we deduce that $D_E$ is isomorphic to a corner of $D_F$. Lastly, we show that this isomorphism implements the homeomorphism between the boundary-path spaces.
\end{abstract}
\maketitle

\section*{Introduction}
Cuntz and Krieger introduced and studied $C^*$-algebras associated to finite $(0,1)$-matrices in~\cite{CK1980}. Within a year, Enomoto and Watatani showed in~\cite{EW1980} how to interpret the Cuntz-Krieger relations and the hypotheses of Cuntz and Krieger's main theorems very naturally in terms of directed graphs. This opened many doors to operator algebraists: graph $C^*$-algebras have provided a rich supply of very tractable examples. In particular, the combinatorial properties of a graph are strongly tied to the algebraic properties of its $C^*$-algebra. Graph $C^*$-algebras include (up to Morita equivalence) all AF algebras~\cite{Drinen2000} and all Kirchberg algebras with free abelian $K_1$~\cite{Szyma'nski2002a}, as well many non-simple examples of purely infinite nuclear $C^*$-algebras.

The original analyses graph $C^*$-algebras utilised the powerful theory of groupoid $C^*$-algebras \cite{Renault1980}. In \cite{KPRR1997}, Kumjian, Pask, Raeburn and Renault built a groupoid $\Gg_E$ from each directed graph $E$, then using Renault's theory of groupoid $C^*$-algebras, they defined the graph $C^*$-algebra to be the groupoid $C^*$-algebra $C^*(\Gg_E)$. By interpreting Renault's hypotheses in terms of the graph $E$ from which $\Gg_E$ was built, they were able to link properties of $E$ to those of $C^*(\Gg_E)$. The analysis of \cite{KPRR1997} establishes among other things that $C^*(\Gg_E)$ is the universal $C^*$-algebra generated by a collection of partial isometries satisfying relations now known as the Cuntz-Krieger relations (Section~\ref{1graph alg}).

The results of \cite{KPRR1997} were proved only for graphs in which each vertex emits and receives only finitely many edges. A significantly different way to construct $\Gg_E$ was introduced by Paterson in \cite{Paterson2002}. Paterson's construction proceeds via inverse semigroups, and provides a framework for a groupoid-based analysis of the graph algebras of directed graphs which may contain infinite receiving vertices. Common to both groupoid models is that the locally compact Hausdorff unit space $\Gg_E^0$ of the groupoid is a collection of paths in the graph: for a row-finite graph with no sources, $\Gg_E^0$ is the collection of right-infinite paths in $E$; but for more complicated graphs, the infinite paths are replaced with the \emph{boundary paths}. Hence the path space of a graph as a topological space is of great importance in the context of graph $C^*$-algebras.

Drinen and Tomforde~\cite{DT2005} construct from an arbitrary directed graph $E$ a row-finite graph $F$ such that $C^*(F)$ contains $C^*(E)$ as a full corner. Their construction adds an infinite path to each source and each infinite receiver in $E$. In the case of infinite receivers, the incoming edges are distributed along the appended infinite path. The resulting graph $F$ is called a \emph{Drinen-Tomforde desingularisation of $E$}. At an infinite receiver, there is a choice in the way which edges are distributed along the appended path, and hence a Drinen-Tomforde desingularisation of $E$ is not unique. Motivated by~\cite{DT2005}, Raeburn developed a `collapsing' technique in~\cite[Section~5]{Raeburn2005} which we use in this paper. He defined a desingularisation by identifying paths in a row-finite graph $F$ with no sources which we call \emph{collapsible paths} (Definition~\ref{colpathdef}), then `collapsed' these paths to yield a graph $E$ such that by applying Drinen and Tomforde's construction (and making the right choices along the way), we can recover $F$.

This paper is an expos\'e of path spaces of directed graphs, and how they are affected by desingularisation. We begin in Section~\ref{1graph chap} by recalling the standard definitions and notation for directed graphs, their $C^*$-algebras, and define a Drinen-Tomforde desingularisation.

In Section~\ref{1graph top} we construct a topology on the path space of an arbitrary directed graph $E$, and show that it is a locally compact Hausdorff topology. Although such results can already be found in the literature, arguments are not provided in this generality. Our construction follows the approach of Paterson and Welch~\cite{PW2005}, and we fix a minor oversight in their work. We construct the homeomorphism $\phi_\infty$, which identifies a subset of the infinite-path space of a desingularisation with the boundary-path space in the original graph.

In Section~\ref{1graph diag}, we define the diagonal $C^*$-subalgebra of a graph $C^*$-algebra. We then build the homeomorphism $h_E$ between the boundary-path space $\partial E$ of an arbitrary graph $E$ and the spectrum of its diagonal. We show that for a desingularisation $F$ of $E$, the isomorphism which embeds $C^*(E)$ as a full corner in $C^*(F)$ implements the homeomorphism $\phi_\infty$ constructed in Section~\ref{1graph top} via the homeomorphisms $h_E$ and $h_F$.

\subsection*{Acknowledgements}
The work contained in this paper is from the author's PhD thesis; as such I extend many thanks to my PhD supervisors Iain Raeburn and Aidan Sims for their support and willingness to proofread and guide my work.

\section{Preliminaries}\label{1graph chap}

A \emph{directed graph} $E = (E^0, E^1, r, s)$ consists of two countable sets $E^0$, $E^1$ and functions
$r,s: E^1 \to E^0$. The elements of $E^0$ are called \emph{vertices} and the elements of $E^1$ are called \emph{edges}. For each edge $e$, we call $s(e)$ the \emph{source} of $e$ and $r(e)$ the \emph{range} of $e$; if $s(e) = v$ and $r(e) = w$, we say that $v$ \emph{emits} $e$ and that $w$ \emph{receives} $e$, or
that $e$ is an edge from $v$ to $w$. Since all graphs in this paper are directed, we often just call a directed graph $E$ a graph.

We follow the convention of~\cite{Raeburn2005}, so that a \emph{path of length} $n$ in a directed graph $E$ is a sequence $\mu = \mu_1 \dots \mu_n$ of edges in $E$ such that $s(\mu_i) = r(\mu_{i+1})$ for $1 \leq i \leq n-1$,
%\begin{center}
%  \begin{tikzpicture}
%    \node (v1) at (0,0) {$\bullet$};
%    \node (v2) at (2,0) {$\bullet$}
%        edge[->] node[auto] {$\mu_1$} (v1);
%    \node (v3) at (4,0) {$\bullet$}
%        edge[->] node[auto] {$\mu_2$} (v2);
%    \node (v4) at (6,0) {$\bullet$};
%        %edge[dotted] (v3);
%    \node (edge) at (5,0) {$\dots$};
%    \node (v5) at (8,0) {$\bullet$}
%        edge[->] node[auto] {$\mu_n$} (v4);
%
%  \end{tikzpicture}
%\end{center}
%
%
We write $|\mu| = n$ for the length of $\mu$, and regard vertices as paths of length $0$; we denote by $E^n$ the set of paths of length $n$, and define $E^* := \bigcup_{n\in\NN} E^n$. We extend the range and source maps to $E^*$ by setting $r(\mu) = r(\mu_1)$ and $s(\mu) = s(\mu_{|\mu|})$ for $|\mu| >1$, and $r(v) = v = s(v)$ for $v \in E^0$. If $\mu$ and $\nu$ are paths with $s(\mu) = r(\nu)$, we write $\mu\nu$ for the path $\mu_1 \dots \mu_{|\mu|} \nu_1 \dots \nu_{|\nu|}$. For a set of vertices $V \subset E^0$ and a set of paths $F \subset E^*$, we define $VF:= \{\mu \in F: r(\mu) \in V\}$ and $FV := \{\mu \in F: s(\mu) \in V\}$. If $V  = \{v\}$ we drop the braces and write $vF$ and $Fv$. We define the infinite paths $E^\infty$ of $E$ to be infinite strings $\mu_1 \dots \mu_n \dots$ such that $s(\mu_i) = r(\mu_{i+1})$ for all $i \geq 1$, we extend the range map to $E^\infty$ by setting $r(\mu) = r(\mu_1)$, and for a set of vertices $V \subset E^0$, we define $VE^\infty:=\{x \in E^\infty : r(x) \in V\}$.

If $r^{-1}(v)$ is finite for every $v \in E^0$, we say that $E$ is \emph{row-finite}. A vertex $v$ is \emph{singular} if $|r^{-1}(v)| \in \{0,\infty\}$. The \emph{boundary paths} of $E$ are defined by $\partial E := E^\infty \cup \{ \alpha \in E^*: s(\alpha)\text{ is singular}\}$.

\subsection{Graph $C^*$-algebras}\label{1graph alg}
Let $E$ be a directed graph. Define
\[
    E^{\leq n} := \{\mu \in E^* : |\mu| = n,\text{ or }|\mu|<n\text{ and } s(\mu)E^1=\emptyset\}.
\]

A \emph{Cuntz-Krieger $E$-family} consists of mutually orthogonal projections $\{s_v : v \in E^0\}$ and partial isometries $\{s_\mu : \mu \in E^*\}$ such that $\{s_\mu : \mu \in E^{\leq n}\}$ have mutually orthogonal ranges for each $n\in \NN$, and such that
\begin{enumerate}
\renewcommand{\theenumi}{CK\arabic{enumi}}
\item $s_\mu^*s_\mu = s_{s(\mu)}$ for every $\mu\in E^*$\label{dgck1};
\item $s_\mu s_\mu^* \leq s_{r(\mu)}$ for every $\mu \in E^*$\label{dgck2}; and
\item $s_v = \sum_{\nu\in vE^{\leq n}} s_\nu s_\nu^*$ for every $v \in E^0$ and $n \in \NN$ such that $|vE^{\leq n}| < \infty$\label{dgck3}.
\end{enumerate}
The $C^*$-algebra of $E$ is the universal $C^*$-algebra $C^*(E)$ generated by a Cuntz-Krieger $E$-family $\{s_\mu : \mu \in E^*\}$. The existence of such a $C^*$-algebra follows from an argument like that of~\cite[Proposition~1.21]{Raeburn2005}.

These relations are slightly different to the Cuntz-Krieger relations appearing elsewhere (for example in~\cite{BHRS2002, DT2005, Raeburn2005}), but straightforward calculations show that our definition is equivalent to the one usually stated. For details refer to~\cite[Section~2.3]{WebsterPhD}.

\subsection{Desingularisation}
Let $\mu \in E^\infty$ and $e \in E^1$. We say that \emph{$e$ exits $\mu$} if there exists $i \geq 1$ such that $s(e) = s(\mu_i)$ and $e \neq \mu_i$; note that edges with source $r(\mu)$ are not considered exits of $\mu$. We say that \emph{$e$ enters $\mu$} if there exists $i \geq 1$ such that $r(e) = r(\mu_i)$ and $e \neq \mu_i$.
\begin{defn}\label{colpathdef}

Let $E$ be a directed graph. We say that an infinite path $\mu \in E^{\infty}$ is \emph{collapsible} if
\begin{enumerate}
\renewcommand{\theenumi}{C\arabic{enumi}}
    \item $\mu$ has no exits\label{c1},
    \item $r^{-1}(r(\mu_i))$ is finite for every $i$,\label{c2}
    \item $r^{-1}(r(\mu)) = \{ \mu_1 \}$\label{c3},
    \item $\mu_i \neq \mu_j$ for all $i \neq j$, and\label{c4}
    \item $\mu$ has either zero or infinitely many entries.\label{c5}
\end{enumerate}
\end{defn}

In~\cite[p42]{Raeburn2005} only~\eqref{c1}--\eqref{c3} are present. Condition~\eqref{c4} was added after we realized that a cycle with no entrance could be collapsible under the original definition, and~\eqref{c5} was added to ensure that we only collapse paths (a process described in Remark~\ref{howtocollapse}) which yield singular vertices - thus avoiding a complication in the proof of~\cite[Proposition~5.2]{Raeburn2005}\footnote{The proof of~\cite[Proposition~5.2]{Raeburn2005} contained an error when proving that the Cuntz-Krieger relation holds in $F_\mu$ at the vertex resulting from collapsing a path $\mu$ in with finitely many entries.}, the key result for this theory. These conditions are not all necessary to carry out the process of collapsing, but they ensure the simplest formulae, and also that we collapse as few paths as possible.

\begin{rmk}\label{howtocollapse}\label{desingpaths}
As the name suggests, we will collapse these paths to form a new graph. Suppose that $\mu$ is a collapsible path in a row-finite graph $F$. Define $s_{\infty}(\mu):= \{ s(\mu_i) : i \geq 1\}$ and
\[
    F^*(\mu) := \{ \nu \in F^* : |\, \nu| > 1, \nu = \mu_1
\mu_2\ \dots\mu_{|\,\nu-1|}e\text{ for some } e \neq \mu_{|\nu|}\}.
\]
Set $F_\mu^0:= F^0 \setminus s_{\infty}(\mu)$ and $F_\mu^1 := \big( F^1 \setminus (r^{-1}(s_{\infty}(\mu))\cup\{\mu_1\})\big)\cup\{e_\nu : \nu \in F^*(\mu)\},$ and extend the range and
source maps to $F_\mu^1$ by setting $r(e_\nu):=r(\nu)=r(\mu)$ and $s(e_\nu):=s(\nu)$. Then $F_\mu$ is the graph obtained by collapsing the path $\mu$ in $F$. Notice that for $\alpha \in F_\mu^*$, $s(\alpha)$ is singular if and only if $s(\alpha) = r(\mu)$.

Given a collection $M$ of collapsible paths such that no two paths in $M$ have any edge or vertex in common, we call the paths in $M$ \emph{disjoint}. We can carry out the process described in Remark~\ref{howtocollapse} on all the paths in $M$ simultaneously, yielding a graph $F_M$ which may no longer be row-finite.

\end{rmk}

\begin{example}
Collapsing the path $\nu_3\nu_4\dots$ in the graph on the left yields the graph on the right.

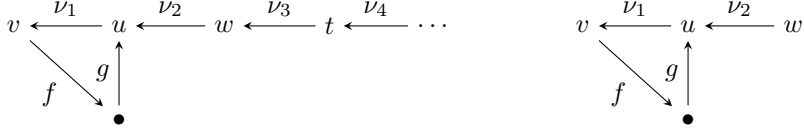
\begin{tikzpicture}[>=stealth,scale=0.7]
    \node (v) at (0,0) {$v$};
    \node (u) at (2,0) {$u$}
        edge[->] node[auto,swap] {$\nu_1$} (v);
    \node (w) at (4,0) {$w$}
        edge[->] node[auto,swap] {$\nu_2$} (u);
    \node (t) at (6,0) {$t$}
        edge[->] node[auto,swap] {$\nu_3$} (w);
    \node (z) at (8,0) {$\dots$}
        edge[->] node[auto,swap] {$\nu_4$} (t);
    \node (f) at (2,-1.8) {$\bullet$}
        edge[<-] node[auto] {$f$} (v)
        edge[->] node[auto] {$g$} (u);
\end{tikzpicture}\qquad\qquad
\begin{tikzpicture}[>=stealth,scale=0.7]
    \node (v) at (0,0) {$v$};
    \node (u) at (2,0) {$u$}
        edge[->] node[auto,swap] {$\nu_1$} (v);
    \node (w) at (4,0) {$w$}
        edge[->] node[auto,swap] {$\nu_2$} (u);
    \node (f) at (2,-1.8) {$\bullet$}
        edge[<-] node[auto] {$f$} (v)
        edge[->] node[auto] {$g$} (u);
\end{tikzpicture}

Notice that the path $(\nu_1gf)^\infty := \nu_1gf\nu_1gf\dots$ is not collapsible as it fails \eqref{c4}, and $\nu_1\nu_2\dots$ is not collapsible either as it has exactly one entry, failing \eqref{c5}.
\end{example}

\begin{defn}\label{DTdesing}
Let $E$ be a directed graph. A \emph{Drinen-Tomforde desingularisation} of $E$ is a pair $(F,M)$ consisting of a row-finite graph $F$ with no sources, and a collection $M$ of disjoint collapsible paths such that $F_M \cong E$.
\end{defn}

\section{Topology}\label{1graph top}
For $\mu \in E^*$, we define the \emph{cylinder set of $\mu$} by $\Zz(\mu) := \{ \nu \in E^* \cup E^\infty : \nu = \mu\nu'\}.$ Following Paterson and Welch's approach in~\cite{PW2005}, define $\alpha:E^*\cup E^\infty \to \{0,1\}^{E^*}$ by $\alpha(w)(y) = 1$ if $w \in \Zz(y)$, and $0$ otherwise. We endow $\{0,1\}^{E^*}$ with the topology of pointwise convergence, and $W$ with the initial topology induced by $\{\alpha\}$. The following Theorem is considered a folklore result, for which we provide a proof.

\begin{theorem}\label{directed graph topology}
Let $E$ be a directed graph. For $\mu \in E^*$ and a finite subset $G \subset s(\mu)E^1$, define $\Zz(\mu\setminus G) := \Zz(\mu) \setminus \bigcup_{e \in G} \Zz(\mu e)$. Then the collection
\[
    \{ \Zz(\mu\setminus G): \mu \in E^*, G \subset s(\mu)E^1\text{ is finite}\}
\] is a basis for the initial topology induced by $\{\alpha\}$. Moreover, it is a locally compact Hausdorff topology on $E^* \cup E^\infty$.
\begin{proof}
First we consider the topology on $\{0,1\} ^{E^*}$. Given disjoint finite subsets $F,G \subset
E^*$, define sets $U_\mu^{F,G}$ to be $\{1\}$ if $\mu \in F$, $\{0\}$ if $\mu \in G$ and $\{0,1\}$ otherwise. Then the sets $N(F, G):= \prod_{\mu \in E^*}U_\mu^{F,G},$ where $F,G$ range over all finite, disjoint pairs of subsets of $E^*$, form a basis for the topology on $\{0,1\}^{E^*}$. Clearly, $\alpha$ is a homeomorphism onto its range, hence the sets $\alpha^{-1}(N(F,G))$ form a basis for a topology on $E^* \cup E^\infty$. Observe that
\[
\alpha^{-1}(N(F, G))= \lambda \in \left(\bigcap_{\mu \in F} \Zz(\mu)\right) \setminus \left(\bigcup_{\nu \in G} \Zz(\nu)\right).
\]
Notice that if $\alpha^{-1}(N(F, G))$ is non empty, then $\bigcap_{\mu \in F}
\Zz(\mu) \neq \emptyset$. This implies that for $\mu,\nu \in F$, we have either
\[
\mu \in \Zz(\nu)\text{ if }|\mu| \geq |\nu|\text{,\quad or }\nu \in \Zz(\mu)\text{ if }|\nu| > |\mu|.
\]
By choosing $\mu$ such that $|\mu| = \max\{| \nu |:\nu \in F\}$ and appropriately adjusting $G$, we see that each $\alpha^{-1}(N(F,G))$ has the form $\Zz(\mu \setminus G)$ for some $\mu \in E^*$ and finite $G \subset s(\mu)E^*$.

\begin{claim}
$\{\Zz(\mu \setminus G):\mu \in E^*, G \subset s(\mu) E^1 \text{ is finite}\}$ and $\{\Zz(\mu \setminus G): \mu \in E^*, G \subset s(\mu) E^* \text{ is finite}\}$ are bases for the same topology.
\begin{proof}
Fix $\mu \in E^*$, and a finite subset $G \subset s(\mu)E^*$. Let $\lambda \in \Zz(\mu \setminus G)$. We seek $\alpha \in E^*$ and a finite set $F \subset s(\alpha)E^1$ such that
\[
\lambda \in \Zz(\alpha \setminus  F) \subset \Zz(\mu \setminus  G).
\]
We consider two cases: $\lambda$ is finite or $\lambda$ is infinite. If $\lambda \in E^\infty$, let $N=\max \{ |\mu\nu| : \nu \in G\}$, $\alpha = \lambda_1 \dots \lambda_N$, and $F = \emptyset$. Then $\Zz(\alpha \setminus F)=\Zz(\alpha)$ clearly contains $\lambda$. Since $|\alpha| \geq |\mu\nu|$ for all $\nu \in G$, we have $\Zz(\alpha) \subset
\Zz(\mu\setminus G)$ as required.

Now suppose that $\lambda \in E^*$. Set $\alpha = \lambda$ and
\[
    F = \{ (\mu\nu)_{|\lambda| +1} : \nu \in G \text{ satisfies } |\mu\nu| > |\lambda| \}.
\]
Then $\Zz(\alpha \setminus F) = \Zz(\lambda \setminus F)$ clearly contains $\lambda$. To see that $\Zz(\lambda\setminus F) \subset \Zz(\mu \setminus G)$, fix $\beta \in \Zz(\lambda \setminus F ).$ Factor $\lambda = \mu\lambda'$, then we have $\beta = \lambda\beta' = \mu\lambda'\beta' \in \Zz(\mu)$. We now show that $\lambda'\beta' \notin \bigcup_{\nu\in G} \Zz(\nu)$. Fix $\nu \in G$. If $|\mu\nu| \leq |\lambda|$, then $|\nu| \leq |\lambda'|$. Since $\lambda' \notin \Zz(\nu)$, we have $\lambda' \beta' \notin \Zz(\nu)$. If $|\mu\nu| > |\lambda|$, then since $\beta'_1 \notin F$, we have $(\mu\lambda'\beta')_{|\lambda| +1} = \beta'_1 \neq (\mu\nu)_{|\lambda|+1}$. So $(\lambda'\beta')_{|\lambda| - |\mu| + 1} \neq \nu_{|\lambda| - |\mu| + 1}$.\pc
\end{proof}
\end{claim}
So the collection $\{ \Zz(\mu\setminus G) : \mu \in E^*, G \subset s(\mu)^1\text{ is finite}\}$ is a basis for our topology on $E^* \cup E^\infty$.

To see that $E^* \cup E^\infty$ is a locally compact Hausdorff space, we follow the strategy of~\cite{PW2005} to show that $\Zz(v)$ is compact for each $v \in E^0$. Since $\alpha$ is a homeomorphism onto its range, it suffices to prove that $\alpha(\Zz(v))$ is compact. Since $\{0,1\}^{E^*}$ is compact, we show that $\alpha(\Zz(v))$ is closed.

Let $\{\omega^{(n)} \in \Zz(v) : n \in \NN\}$ be such that $\alpha(\omega^{(n)}) \to f \in \{0,1\}^{E^*}$. We seek $\omega \in \Zz(v)$ such that $f = \alpha(\omega)$. Let $A:=\{\mu \in E^* : \alpha(\omega^{(n)})(\mu) \to 1\}.$ Then if $\mu,\nu \in A$, for large $n$ we have that $w^{(n)} \in \Zz(\mu) \cap \Zz(\nu)$. In particular, $\Zz(\mu) \cap \Zz(\nu) \neq \emptyset$; without loss of generality say $\mu = \nu\nu'$, and denote it $\beta_{\mu,\nu}$. Then for large $n$ we have that $\omega^{(n)} \in \Zz(\beta_{\mu,\nu})$, so $\beta_{\mu,\nu} \in A$.

Since $A$ is countable, we can list $A = \{\nu^1,\nu^2,\dots,\nu^m,\dots\}.$ Let $y^1:=\nu^1$, and iteratively define $y^n := \beta_{y^{n-1},\nu^n}$. Then $\{y^n : n \in \NN\}$ satisfy $y^n_1 y^n_2 \dots y^n_{|y^{n-1}|} = y^{n-1}$, and hence they determine a unique path $\omega \in E^*\cup E^\infty$.

To see that $\alpha(\omega^{(n)}) \to \alpha(\omega)$, we first show that $\nu \in A$ if and only if $\omega \in \Zz(\nu)$. Clearly, $\omega \in \Zz(y^m) \subset \Zz(\nu^m)$ for each $\nu^m \in A$. Conversely, let $\omega \in \Zz(\nu^m)$. Then $y^m \in \Zz(\nu^m) \cap A$ implies that for large enough $n$ we have $\omega^{(n)} \in \Zz(y^m) \subset \Zz(\nu^m)$, so $\nu^m \in A$. Now fix $\nu \in E^*$. We will show that $\alpha(\omega^{(n)})(\nu) \to \alpha(\omega)(\nu)$. If $\alpha(\omega)(\nu) = 1$, then $\omega \in \Zz(\nu)$. So $\nu \in A$, and hence $\omega^{(n)}(\nu) \to 1$. If $\alpha(\omega)(\nu) = 0$, we have $\omega \notin \Zz(\nu)$, forcing $\alpha(\omega^{(n)})(\nu) \to 0$. So $\alpha(\omega^{(n)}) \to \alpha(\omega)$. Hence $\alpha(\Zz(v))$ is closed.
\end{proof}
\end{theorem}

\begin{theorem}\label{dg boundary is top inv under desing}
Let $E$ be a directed graph and $F$ be a Drinen-Tomforde desingularisation of $E$. Then $E^0F^\infty$ is homeomorphic to $\partial E$.
\end{theorem}

%To prove Theorem~\ref{dg boundary is top inv under desing}, we define a map $\phi$ (Equation~\eqref{graph desing finite path bij}) on finite paths in $F$ with range and source in $E$. Using Lemma~\ref{fhatpaths}, we then use $\phi$ to define a map $\phi_\infty:E^0F^\infty \to \partial E$ as in~\eqref{graph desing boundary path bij}, which we prove is a homeomorphism.

Suppose $E$ is a directed graph, and $(F,M)$ is a Drinen-Tomforde desingularisation of $E$. Define
$F^*(M) := \bigcup_{\mu \in M} F^*(\mu).$ Define $\phi' : (F^1 \cap E^1)\cup F^*(M) \to E^1$ by $\phi'|_{F^1 \cap E^1} := \id_{F^1 \cap E^1}$ and $\phi'|_{F^*(M)} : \nu \mapsto e_\nu.$ So $\phi'$ acts as the identity on unchanged edges, and takes collapsible paths in $F$ to the associated edges in $E$.

If $\beta \in F^*$ with $r(\beta), s(\beta) \in E^0$, then $\beta$ has the form $\beta = b^1 b^2 \dots
b^n$ where each $b^k \in (F^1 \cap E^1) \cup F^*(M)$. Define $E^0F^*E^0:= \{ \beta \in F^*: r(\beta),s(\beta) \in E^0 \}$. We extend the map $\phi'$ above to a map $\phi$ on finite paths: define $\phi: E^0F^*E^0 \to E^*$ by
\begin{equation}\label{graph desing finite path bij}
\phi (\beta) := \phi(b^1b^2 \dots b^n) = \phi'(b^1)\dots\phi'(b^n).
\end{equation}
We will extend this map to $E^0F^\infty$, and ultimately show that it is a homeomorphism from $E^0F^\infty$ to $\partial E$. To do so precisely we use the following results.

\begin{lemma}\label{fhatpaths}
Let $E$ be a directed graph, and $(F,M)$ be a desingularisation of $E$. If $\lambda \in E^0F^\infty$, then either
\begin{itemize}
\item $\lambda = l^1 \dots l^k \mu$ for some $\mu \in M$ and $l^i \in (F^1 \cap E^1) \cup F^*(M)$, or
\item $\lambda = l^1 l^2 \dots l^n \dots$ where $l^i \in (F^1 \cap E^1) \cup F^*(M)$.
\end{itemize}

\begin{proof}
Fix $\lambda \in E^0F^\infty$. We construct the $l^i$ inductively. Either $\lambda_1 \in F^1 \cap E^1$, or $\lambda_1 = \mu_1$ for some $\mu \in M$. If $\lambda_1 \in F^1
\cap E^1$, then let $l^1 = \lambda_1$. If $\lambda_1 = \mu_1$, then either
\begin{enumerate}
\renewcommand{\theenumi}{\roman{enumi}}
\item $\lambda_i = \mu_i$ for all $i \in \NN$, in which case $\lambda = \mu$; or
\item there exists $k$ such that
%$\lambda_1\dots\lambda_k \in F^*(\mu)$, in which case set $l^1 = \lambda_1 \dots \lambda_k$.
    $\lambda_i = \mu_i$ for all $i < k$ and $\lambda_k \neq \mu_k$, in which case we set $l^1 = \mu_1 \dots \mu_{k-1} \lambda_k$. Since paths in $M$ have no edges in common, we have $l^1 \in F^*(\mu)$.
\end{enumerate}
In case (i). $\lambda = \mu$, in which case we are done. In case (ii), $\lambda = l^1\lambda'$ for some $\lambda' \in F^\infty$. Iterating will either terminate with $\lambda = l^1 \dots l^n \mu$ where $\mu \in M$, or continue ad infinitum, in which case $\lambda = l^1 \dots l^n \dots$.
\end{proof}
\end{lemma}

Define $\phi_{\infty}: E^0F^\infty \to \partial E$ by
\begin{equation}\label{graph desing boundary path bij}
\phi_\infty (\lambda) :=
\begin{cases}
    \phi(\lambda') &\text{ if } \lambda = \lambda' \mu \text{ for some }\mu \in M,\\
    \phi'(\lambda^1)\dots \phi'(\lambda^n) \dots &\text{ if } \lambda = l^1 \dots l^n \dots.
\end{cases}
\end{equation}

\begin{prop}[{\cite[Lemma 2.6a]{DT2005}}]\label{phibij}
Let $E$ be a directed graph, and $(F,M)$ be a desingularisation of $E$. Then $\phi$ and $\phi_\infty$, defined as in~\eqref{graph desing finite path bij} and~\eqref{graph desing boundary path bij} respectively, are bijections and preserve range and source.
\end{prop}

\begin{rmk}When working with the topology on the infinite path space of a  row-finite directed graph $F$ with no sources, the finite compliments are unnecessary~\cite[Corollary~2.2]{KPRR1997}. For a detailed proof of this statement, see the author's PhD thesis~\cite[Proposition~2.1.2]{WebsterPhD}.
\end{rmk}

\begin{proof}[Proof of Theorem~\ref{dg boundary is top inv under desing}.]
It suffices to show that $\phi_\infty$ and $\phi_\infty^{-1}$ are continuous.

To see that $\phi_\infty$ is continuous, fix a basic open set $\Zz(\alpha \setminus G) \cap \partial E$. If $\Zz(\alpha \setminus G) \cap \partial E = \emptyset$ then $\phi_\infty^{-1}(\Zz(\alpha \setminus G) \cap \partial E) = \emptyset$ is open. Suppose that $\Zz(\alpha \setminus G) \cap \partial E \neq \emptyset$, and fix $\lambda \in \phi_\infty^{-1}(\Zz(\alpha \setminus G) \cap \partial
E)$. We seek $\gamma \in F^*$ such that
\[
    \lambda \in \Zz(\gamma) \cap E^0F^\infty \subset \phi_\infty^{-1}(\Zz(\alpha \setminus G) \cap \partial E).
\]
We consider two cases:
\begin{enumerate}
\renewcommand{\theenumi}{\roman{enumi}}
\item Either $\lambda = l^1 l^2 \dots$, or $\lambda = l^1 \dots l^k \mu$ with $k > |\alpha|$; and
\item $\lambda = l^1 \dots l^{|\alpha|} \mu$.
\end{enumerate}
where $\mu \in M$, and $l^i \in (F^1 \cap E^1) \cup F^*(M)$ for each $i$.

In case (i), let $\gamma = l^1 \dots l^{|\alpha|+1}$. Clearly $\lambda \in \Zz(\gamma)
\cap E^0F^\infty$. Furthermore, for $y \in \Zz(\gamma) \cap E^0F^\infty$, $\phi'(l^1)\dots \phi'(l^{|\alpha|}) = \alpha$ and $\phi'(l^{|\alpha|+1}) \notin G$.  So $\phi_\infty (y) = \phi_\infty(l^1 \dots l^{|\alpha|+1} y') \in \Zz(\alpha \setminus G) \cap \partial E.$

In case (ii), we have that $s(\alpha)$ is singular in $E$. Since $G \subset s(\alpha)E^1$,~\eqref{c3} implies that $G \subset \phi (F^*(M))$. Let $N = \max_{\nu \in \phi^{-1}(G)}|\nu|$. Each $\nu \in G \cap E^N$ has the form $\mu_1 \dots \mu_{N-1} e$, where $e \neq \mu_N$. Set $\gamma = \phi^{-1}(\alpha)\mu_1 \dots \mu_N$. Then $\lambda = \phi^{-1}(\alpha)\mu \in \Zz(\gamma)\cap E^0F^\infty$, and $\Zz(\gamma) \cap E^0F^\infty \subset \phi_\infty^{-1}(\Zz(\alpha \setminus G) \cap \partial E$.

To see that $\phi_\infty^{-1}$ is continuous, a basic open set $\Zz(\gamma) \cap E^0F^\infty$ in $E^0F^\infty$. If $\Zz(\gamma) \cap E^0F^\infty= \emptyset$ then $\phi_\infty
(\Zz(\gamma) \cap E^0F^\infty) = \emptyset$ is open, so suppose that $\Zz(\gamma) \cap E^0F^\infty \neq
\emptyset$. Let $x \in \phi_\infty(\Zz(\gamma)\cap E^0F^\infty)$. We seek $\alpha \in E^*$ and a
finite subset $G \subset s(\alpha)E^1$ such that
\[
    x \in \Zz(\alpha \setminus G) \cap \partial E \subset \phi_\infty(\Zz(\gamma)\cap E^0F^\infty).
\]
Let $\lambda = \phi_\infty^{-1}(x)= \gamma \lambda'$ where $\lambda' \in F^\infty$. We consider two cases:
\begin{enumerate}
\renewcommand{\theenumi}{\roman{enumi}}
\item $x \in E^\infty$, or
\item $x \in E^*$ and $s(x)$ is singular.
\end{enumerate}

In case (i), $\lambda$ does not `start' with a collapsible path, so by Lemma~\ref{fhatpaths} $\lambda = l^1l^2\dots$ for some $l^i \in (E^1 \cap F^1)\cup F^*(M)$. Let $j = \min\{i\in\NN:|l^1\dots l^i| \geq |\gamma|\}$, set $\alpha = \phi(l^1 \dots l^j)$ and $G=\emptyset$. It follows that $x \in \Zz(\alpha) \cap \partial E \subset \phi_\infty(\Zz(\gamma) \cap E^0F^\infty).$

In case (ii), we have $\lambda = \gamma \lambda' = \omega\mu$ for some $\omega \in F^*$ and $\mu \in M$. Let $\alpha:=x$. Our choice of $G$ depends on $|\gamma|$, so we argue in cases:
\begin{enumerate}
\item If $|\gamma| \leq |\omega|$, let $G = \emptyset$.
\item If $|\gamma| > |\omega|$, then $\gamma = \omega\mu_1 \dots \mu_j$ for some $j \in \NN$; let
\[
    G = \{ e_\nu : \nu = \mu_1 \dots \mu_k \nu_{k+1} \in F^*(\mu)\text{, and } k <j\}.
\]
\end{enumerate}
Since $x \in \Zz(\alpha \setminus G) \cap \partial E$ by definition, we just need to show that
\[
    \Zz(x \setminus G) \cap \partial E \subset \phi_\infty(\Zz(\gamma) \cap E^0F^\infty).
\]
Fix $y=xy' \in \Zz(x \setminus G) \cap \partial E$. Since $x = \phi_\infty(\omega \mu) = \phi(\omega)$, we have $\phi_\infty^{-1}(y) = \omega\phi_\infty^{-1}(y').$ In case (1), $|\gamma| \leq |\omega|$ implies that $\omega = \gamma \omega'$ for some $\omega' \in F^*$, so
\[
\phi_\infty^{-1}(y) = \gamma\omega'\phi_\infty^{-1}(y') \in \Zz(\gamma) \cap E^0F^\infty.
\]
For case (2), observe that if $y'\in E^0$, then $y=x \in \phi_\infty(\Zz(\gamma) \cap E^0F^\infty)$ by
assumption. Suppose that $|y'|\geq 1$. Then $y'_1 = e_\nu$ for some $\nu \in F^*(\mu)$. Since $y \in \Zz(x\setminus G)$, $y'_1 \notin G$, so $\nu = \mu_1
\dots \mu_k \nu_{k+1}$ for some $k \geq j$, and thus
\[
\phi_\infty^{-1}(y) = \phi_\infty^{-1}(xy') = \omega \nu \phi_\infty^{-1}(y'_2\dots) = \gamma \mu_{j+1} \dots \mu_k \nu_{k+1} \phi_\infty^{-1}(y'_2\dots)
\]
is an element of $\Zz(\gamma) \cap E^0F^\infty.$ So $y \in \phi_\infty( \Zz(\gamma) \cap E^0F^\infty)$, and hence $\phi_\infty: E^0F^\infty \to \partial E$ is a homeomorphism.
\end{proof}

\section{The Diagonal and the Spectrum}\label{1graph diag}

For a directed graph $E$, we call $C^*(\{s_\mu s_\mu^* : \mu \in E\}) \subset C^*(E)$ the diagonal $C^*$-algebra of $E$ and denote it $D_E$, dropping the subscript when confusion is unlikely. We denote the spectrum of a commutative $C^*$-algebra $B$ by $\Delta(B)$. Given a homomorphism $\pi:A\to B$ of commutative $C^*$-algebras, we denote by $\pi^*$ the induced map from $\Delta(B)$ to $\Delta(A)$ such that $\pi^*(f)(y) = f(\pi(y))$ for all $f \in \Delta(B)$ and $y \in A$.

\begin{rmk}\label{isom for desing}
Suppose $E$ is a directed graph, and that $(F,M)$ is a Drinen-Tomforde desingularisation of $E$. Let $\{s_\mu:\mu \in E^*\}$ and $\{t_\mu : \mu \in F^*\}$ be the Cuntz-Krieger families generating $C^*(E)$ and $C^*(F)$. Then it follows from \cite[Proposition~5.2]{Raeburn2005} that there exists a projection $p$ such that $p C^*(F) p$ is a full corner in $C^*(F)$, and that there is an isomorphism $\pi:C^*(E) \cong pC^*(F)p$ such that $\pi(s_v) = t_v$ for each $v\in E^0$, $\pi(s_\mu) = t_{\phi^{-1}(\mu)}$ for each $\mu \in E^*$.
\end{rmk}

The goal for this section is the following theorem.

\begin{theorem}\label{mainthm1graph}
 Let $E$ be a directed graph and $(F,M)$ be a Drinen-Tomforde desingularisation of $E$. Let $\phi_\infty: E^0 F^\infty \to \partial E$ be the homeomorphism from Theorem~\ref{dg boundary is top inv under desing}, let $p$ and $\pi$ be as in Remark~\ref{isom for desing}. Then $\pi(D_E) = pD_Fp$, and there exist homeomorphisms $h_E:\partial E \to \Delta(D_E)$ and $h:E^0 F^\infty \to \Delta(p D_F p)$ such that the following diagram commutes.
\begin{center}
\begin{tikzpicture}[>=stealth,scale=0.7]
    \node (e0f) at (0,1.5) {$E^0F^\infty$};
    \node (deltapdp) at (0,0) {$\Delta(pD_Fp)$}
        edge[<-] node[auto] {$h$} (e0f);
    \node (partiale) at (4,1.5) {$\partial E$}
        edge[<-] node[auto,swap] {$\phi_\infty$} (e0f);
    \node (deltad) at (4,0) {$\Delta(D_E)$}
        edge[<-] node[auto,swap] {$h_E$} (partiale)
        edge[<-] node[auto] {$\pi^*$} (deltapdp);
\end{tikzpicture}
%\caption{}
\end{center}
\end{theorem}
We prove Theorem~\ref{mainthm1graph} on page \pageref{proofofmainthm1graph}. First, we establish some technical results.

\begin{rmk}\label{ranges are mut orth}
  Let $E$ be a directed graph, and let $\mu,\nu \in E^*$. Then
%  \[
%    s_\mu^* s_\nu = \begin{cases}
%      s_{\nu'}^* \text{ if } \mu = \nu\nu'\\
%      s_{\mu'} \text{ if } \nu = \mu\mu'\\
%      0 \text{ otherwise.}
%    \end{cases}
%  \]
%Furthermore,
\begin{equation}\label{dg product of range projs}
    (s_\mu s_\mu^*)(s_\nu s_\nu^*) = \begin{cases}
      s_\mu s_\mu^* \text{ if } \mu = \nu\nu'\\
      s_\nu s_\nu^*\text{ if } \nu = \mu\mu'\\
      0\text{ otherwise.}
    \end{cases}
\end{equation}
This result is is proved for row-finite directed graphs as~\cite[Corollary~1.14(b)]{Raeburn2005}. The proof is only marginally different for arbitrary directed graphs, for a detailed argument see~\cite[Lemma~2.4.4]{WebsterPhD}.
\end{rmk}

\begin{lemma}\label{L_snusnu*=sumqnu}
Let $E$ be a directed graph, and let $F\subset E^*$ be finite. For $\mu \in F$, define
\[
    q_\mu^F:= s_\mu s_\mu^* \prod_{\mu \mu' \in F\setminus\{\mu\}}(s_\mu s_\mu^* - s_{\mu\mu'} s_{\mu\mu'}^*).
\]
Then the $q_\mu^F$ are mutually orthogonal projections in $\lsp \{s_\mu s_\mu^* : \mu \in F\}$, and for each $\nu \in F$, we have
\begin{equation}\label{eq_snusnu*=sumqnu}
s_\nu s_\nu^* = \sum_{\nu\nu' \in F} q_{\nu\nu'}^F.
\end{equation}
\begin{proof}
By Remark~\ref{ranges are mut orth}, $p:\lambda \to s_\lambda s_\lambda^*$ is a Boolean Representation of $E$ in the sense of~\cite[Definition~3.1]{SW2010}. The result then follows from~\cite[Lemma~3.1]{SW2010}.
\end{proof}
\end{lemma}

\begin{rmk}\label{lemma about projections nonzero}
Let $A$ be a $C^*$-algebra, let $p$ be a projection in $A$, let $Q$ be a finite set of commuting subprojections of $p$ and let $q_0$ be a nonzero subprojection of $p$. Then $\prod_{q\in Q}(p-q)$ is a projection. If $q_0$ is orthogonal to each $q \in Q$, then $q_0\prod_{q\in Q}(p-q) = q_0$, so in particular, $\prod_{q\in Q}(p-q) \neq 0$. The proof if this is relatively simple, details can be found in~\cite[Lemma A.0.7]{WebsterPhD}.
\end{rmk}

\begin{rmk}\label{different form of qmuf}
Let $E$ be a directed graph, and let $F\subset E^*$ be finite. For $\mu \in F$, let $F_\mu =\{\mu' \in s(\mu)E \setminus \{s(\mu)\} : \mu\mu' \in F\}$. It follows from an induction on $|F_\mu|$ that
\[
    q_\mu^F = s_\mu \Big( \prod_{\mu' \in F_\mu}(s_{s(\mu)} - s_{\mu'} s_{\mu'}^*)\Big)s_\mu^*.
\]
\end{rmk}

We say that $\mu,\nu \in E^*$ have \emph{common extension} if either $\mu = \nu\nu'$ or $\nu=\mu\mu'$, and call the longer path the \emph{minimal common extension} of $\mu$ and $\nu$. A set $F \subset E^*$ is \emph{exhaustive} if for every $\mu\in E^*$ there exists $\nu \in F$ such that $\mu$ and $\nu$ have common extension. We denote the set of finite exhaustive sets by $\FE(E)$, and for a vertex $v$ we define $v\FE(E) := \{F \in \FE(E): F \subset vE^*\}$.

\begin{theorem}\label{boundary homeo to spectrum of diagonal}
Let $E$ be a directed graph. Then $D = \clsp\{s_\mu s_\mu^* : \mu \in E\}$, and for each $x \in \partial E$ there exists a unique $h_E(x) \in \Delta(D)$ such that
\[
    h_E(x)(s_\mu s_\mu^*) =
        \begin{cases}
            1 &\text{if $x \in \Zz(\mu)$}\\
            0 &\text{otherwise}.
        \end{cases}
\]
Moreover, $x \mapsto h_E(x)$ is a homeomorphism of $\partial E$ onto $\Delta(D)$.
\begin{proof}
That $D = \clsp \{s_\mu s_\mu^* : \mu \in E^*\}$ follows from equation~\eqref{dg product of range projs}.

Fix $x\in\partial E$ and $\sum_{\mu \in F}b_\mu s_\mu s_\mu^* \in \lsp\{s_\mu s_\mu^* : \mu \in E^*\}$. Let $n=\max\{p \in \NN: x_1\dots x_p \in F\}$, and define $F_x:=\{\mu' \in x(n)E\setminus\{x(n)\} : x(0,n)\mu' \in F\}$.
\begin{claim}\label{qxneq0}
  The projection $q_{x_1\dots x_n}^F \neq 0$.
  \begin{proof}
If $s(x_n)E^* = \emptyset$, then $F_x = \emptyset$, and hence $q_{x_1\dots x_n}^F = s_{x_1\dots x_n}s_{x_1\dots x_n}^* \neq 0.$ Now suppose that $s(x_n)E^* \neq \emptyset$. We first show that there exists $\nu \in s(x_n)E^*$ such that for each $\mu' \in F_x$, $\nu$ and $\mu'$ have no common extension. We argue in cases.
\begin{enumerate}
\renewcommand{\theenumi}{\roman{enumi}}
\item If $s(x)$ is a source in $E$ and $|x| > n$, let $\nu = x_{n+1}\dots x_{|x|}$. Then by choice of $n$, $\nu$ has no common extension with any $\mu'$ in $F_x$.
\item If $s(x)$ is an infinite receiver, such a $\nu$ exists since $|F_x| \leq |F| < |s(x)E^*| = \infty$.
\item If $x \in E^\infty$, let $k= \max\{|\mu'| :\mu'\in F_x \}$. Then it follows from our choice of $n$ that $\nu = x_{n+1}\dots x_{n+k}$ is not a common extension of any $\mu'$ in $F_x$.
\end{enumerate}

By Remark~\ref{ranges are mut orth}, we have $s_\nu s_\nu^* s_{\mu'} s_{\mu'}^* = 0$ for all $\mu' \in F_x$. Applying Lemma~\ref{lemma about projections nonzero} with $p = s_{s(x_n)}$, $q_0 = s_\nu s_\nu^*$, $Q = F_x$, we have $\prod_{\mu'\in F_x}(s_{s(x_n)} - s_{\mu'}s_{\mu'}^*) \neq 0$. So
\[
q_{x_1\dots x_n}^F = s_{x_1\dots x_n} \prod_{\mu'\in F_x}(s_{s(x_n)} - s_{\mu'}s_{\mu'}^*) s_{x_1\dots x_n}^*  \neq 0.\pc\qedhere
\]
\end{proof}
\end{claim}
By the above claim,
\[
\Big\| \sum_{\nu \in F} b_\mu s_\mu s_\mu^* \Big\|
= \Big\| \sum_{\nu \in F} \Big(\sum_{\substack{\mu \in F\\\nu \in \Zz(\mu)}} b_\mu \Big) q^F_\nu \Big\|
= \max_{\substack{\nu \in F\\q_\nu^F \neq 0}} \Big\{\Big|\sum_{\substack{\mu \in F\\\nu \in \Zz(\mu)}} b_\mu\Big|\Big\}
\geq \Big|\sum_{\substack{\mu \in F\\x_1\dots x_n \in \Zz(\mu)}} b_\mu\Big|
\]
Hence the formula
\begin{equation}\label{want h(x) to be this}
h_E(x)\big(\sum_{\mu \in F}b_\mu s_\mu s_\mu^*\big) = \sum_{\substack{\mu\in F\\x \in \Zz(\mu)}} b_\mu
\end{equation}
determines a well-defined, norm-decreasing linear map $h_E(x)$ on $\lsp\{s_\mu s_\mu^* : \mu \in E\}$.

We now show that $h_E(x)$ is a homomorphism. Since $h_E(x)$ is linear and norm-decreasing, it suffices to calculate
\begin{align*}
h_E(x)(s_\mu s_\mu^* s_\alpha s_\alpha^*) &= \begin{cases}1 &\text{if $\alpha \in \Zz(\mu)$ and $x \in \Zz(\alpha)$}\\&\text{or $\mu \in \Zz(\alpha)$ and $x \in \Zz(\mu)$,}\\ 0 &\text{otherwise}\end{cases}\\
&=\begin{cases}1 &\text{if $x \in \Zz(\alpha)\cap\Zz(\mu)$}\\ 0 &\text{otherwise.}\end{cases}\\
&= h_E(x)(s_\mu s_\mu^*) h_E(x)(s_\alpha s_\alpha^*).
\end{align*}
Now $h(x)$ is a nonzero bounded homomorphism on a dense subspace of $D$, and hence extends uniquely to a nonzero homomorphism $h(x):D \to \CC$. It remains to show that $h_E:\partial E \to \Delta(D)$ is a homeomorphism. The trickiest part is to show that $h_E$ is onto.
\begin{claim}
The map $h_E$ is surjective.
\begin{proof}
Fix $\phi \in \Delta(D)$. For each $n \in \NN$, $\{s_\mu s_\mu^* : |\mu| =n \}$ are mutually orthogonal projections, thus there exists at most one $\nu^n \in E^n$ such that $\phi(s_{\nu^n} s_{\nu^n}^*)=1$. Let
\[
S:=\{n\in\NN: \text{ there exists } \nu^n \in E^n \text{ such that } \phi(s_{\nu^n}s_{\nu^n}) = 1\}.
\]
Since $\phi$ is nonzero, $S$ is nonempty. If $\nu = \mu\nu'$ and $\phi(s_\nu s_\nu^*)=1$, then
\[
    1= \phi(s_\nu s_\nu^*) = \phi(s_\nu s_\nu^* s_\mu s_\mu^*) = \phi(s_\nu s_\nu^*) \phi(s_\mu s_\mu^*),
\]
so $\phi(s_\mu s_\mu^*)=1$. It follows that either $S = \NN$, or to $\{1,\dots,N\}$ for some $N$.

If $S = \NN$, define $x\in E^\infty$ by $x(0,n) = \nu^n$ for all $n$. If $S = \{1, \dots, N\}$, define $x:= \nu^N$. That $x \in \partial E$ is trivial if $S=\NN$, and follows from~\eqref{dgck3} otherwise. To see that $h_E(x) = \phi$, notice that for each $\mu \in E^*$ we have
\begin{align*}
    \phi(s_\mu s_\mu^*) = 1 &\iff |\mu| \in S \text{ and }\nu^{|\mu|} = \mu \\&\iff x(0,|\mu|)=\mu \\&\iff h_E(x)(s_\mu s_\mu^*) = 1.
\end{align*}
Since both $\phi(s_\mu s_\mu^*)$ and $h_E(x)(s_\mu s_\mu^*)$ only take values in $\{0,1\}$, it follows that $h_E(x) = \phi$.
\pc\end{proof}
\end{claim}

To see $h$ is injective, suppose that $h_E(x) = h_E(y)$. Then for each $n \in \NN$, let $n_x = \min\{n,|x|\}$. Then we have
\[
h_E(y)(s_{x(0,n_x)} s_{x(0,n_x)}^*) = h_E(x)(s_{x(0,n_x)} s_{x(0,n_x)}^*) = 1
\]
Hence $y(0,n \wedge |x|) = x(0,n \wedge |x|)$ for all $n \in \NN$. By symmetry, we also have that $y(0,n \wedge |y|) = x(0,n \wedge |y|)$ for all $n$. In particular, $|x| = |y|$ and $y(0,n) = x(0,n)$ for all $n \leq |x|$. Thus $x=y$.

Recall that $\Delta(D)$ carries the topology of pointwise convergence. For openness, it suffices to check that $h_E^{-1}$ is continuous. Suppose that $h(x^n) \to h(x)$. Fix a basic open set $\Zz(\mu)$ containing $x$, so $h(x)(s_\mu s_\mu^*) = 1$. Since $h(x^n)(s_\mu s_\mu^*) \in \{0,1\}$ for all $n$, for large enough $n$, we have $h(x^n)(s_\mu s_\mu^*) = 1$. So $x^n \in \Zz(\mu)$. For continuity, a similarly straightforward argument shows that if $x^n \to x$, then $h(x^n)(s_\mu s_\mu^*) \to h(x)(s_\mu s_\mu^*)$. This convergence extends to $\lsp\{s_\mu s_\mu^* : \mu \in E^*\}$ by linearity, and to $D$ by an $\eps/3$ argument.
\end{proof}
\end{theorem}

We can now prove our main result.

\begin{proof}[Proof of Theorem~\ref{mainthm1graph}.]\label{proofofmainthm1graph}
The projection $p$ from Remark~\ref{isom for desing} satisfies
\begin{equation}\label{eqn for proj}
    pt_\mu t_\mu^*p = \begin{cases} t_\mu t_\mu^* &\text{if } r(\mu) \in E^0\\0 &\text{otherwise.}\end{cases}
\end{equation}
We will show that $\pi$ maps $D_E$ onto $p D_Fp$. It follows from~\eqref{eqn for proj} that $\pi(D_E) \subset p D_F p$. To see the reverse inclusion, fix $\mu \in F^*$. If $r(\mu) \notin E^0$ then $p t_\mu t_\mu^* p = 0 \in \pi(D_E)$, so suppose that $r(\mu) \in E^0$. If $s(\mu) \in E^0$, then $p t_\mu t_\mu^* p = t_\mu t_\mu^* = \pi(s_{\phi(\mu)}s_{\phi(\mu)}^*) \in \pi(D_E).$ Now suppose that $s(\mu) \notin E^0$, then $s(\mu) = s(\nu_n)$ for some collapsible path $\nu\in F^\infty$ and $n \in \NN$. Since $\nu$ has no exits except at $r(\nu)$, we have $\mu = \mu' \nu_n$ for $\mu' = \mu(0,|\mu|-1)$. Furthermore, $s(\mu')F^1$ is finite, thus~\eqref{dgck3} implies that
\begin{equation}\label{induction for dgraph thm}
p s_\mu s_\mu^* p = p s_{\mu'} s_{\nu_n} s_{\nu_n}^* s_{\mu'}^* p = p s_{\mu'} s_{\mu'}^* p - \sum_{f \in s(\mu')F^1\setminus\{{\nu_n}\}} p s_{\mu'} s_f s_f^* s_{\mu'}^* p.
\end{equation}
An induction on $n$ gives $p s_{\mu'} s_{\mu'}^* p \in \pi(D_E)$. It then follows from~\eqref{induction for dgraph thm} that $p s_{\mu} s_{\mu}^* p \in \pi(D_E)$, and hence $\pi(D_E) = pD_Fp.$

We now construct the homeomorphism $h$. Since $p$ commutes with $D_F$, the space $p D_F p$ is an ideal of $D_F$. Then~\cite[Propositions A26(a) and A27(b)]{RW1998} imply that the map $k:\phi \mapsto \phi|_{p D_F p}$ is a homeomorphism of $\{\phi\in\Delta(D_F): \phi|_{p D_F p} \neq 0\}$ onto $\Delta(p D_F p).$ Since $F$ has no singular vertices, $\partial F = F^\infty$. Let $h_F:F^\infty \to \Delta(D_F)$ be the homeomorphism obtained from Theorem~\ref{boundary homeo to spectrum of diagonal}. Then $h_F(x) \in \dom(k)$ for all $x \in E^0F^\infty$. Define $h:=k \circ h_F|_{E^0F^\infty}: E^0 F^\infty \to \Delta(p D_F p)$.

We aim to show that $h_E \circ \phi_\infty = \pi^*\circ h$. Let $x \in E^0 F^\infty$, and fix $\mu \in E^*$. Since $(h_E \circ \phi_\infty) (x)$ and $h(x)$ are homomorphisms, and since $\pi$ is an isomorphism, it suffices to show that
\begin{equation}\label{intertwining diagram commutes}
(h_E \circ \phi_\infty)(x) (s_\mu s_\mu^*) = (\pi^*\circ h)(x) (s_\mu s_\mu^*).
\end{equation}
Since $\mu \in E^*$, we have $t_{\phi^{-1}(\mu)} t_{\phi^{-1}(\mu)}^* \in p D_F p$. Then since $r(x) \in E^0$, the right-hand side of~\eqref{intertwining diagram commutes} becomes
\begin{align*}
\pi^* ( h(x)) (s_\mu s_\mu^*) = h(x) (t_{\phi^{-1}(\mu)} t_{\phi^{-1}(\mu)}^*) &= h_F(x)|_{p D_F p} (t_{\phi^{-1}(\mu)} t_{\phi^{-1}(\mu)}^*)\\
&=
\begin{cases}
    1 &\text{if } x \in \Zz(\phi^{-1}(\mu))\\
    0 &\text{otherwise.}
\end{cases}
\end{align*}
We break the left-hand side of~\eqref{intertwining diagram commutes} into cases: (i) $\phi_\infty(x) \in E^\infty$, or (ii) $\phi_\infty(x) \in E^*$.
In case (i), since $\phi_\infty(x) \in \Zz(\mu)$ if and only if $x= \phi^{-1}(\mu)\phi_\infty^{-1}(\mu')$ for some $\mu' \in E^\infty$, the left-hand side of~\eqref{intertwining diagram commutes} becomes
\[
    h_E(\phi_\infty(x))(s_\mu s_\mu^*) =
        \begin{cases}
            1 &\text{if } \phi_\infty(x) \in \Zz(\mu)\\
             0&\text{otherwise}
        \end{cases}
        =
        \begin{cases}
            1 &\text{if } x \in \Zz(\phi^{-1}(\mu))\\
            0&\text{otherwise,}
        \end{cases}.
\]

In case (ii), $\phi_\infty(x) = \phi(x')$, where $x = x'\nu$ for some collapsible path $\nu \in M$. The left hand side of~\eqref{intertwining diagram commutes} then becomes
\[
    h_E(\phi(x'))(s_\mu s_\mu^*) =
    \begin{cases}
     1 &\text{if } \phi(x') \in \Zz(\mu)\\
     0&\text{otherwise.}
    \end{cases}
\]
Since $\phi$ is a bijection, $x' =\phi^{-1}(\mu) x''$ if and only if $\phi(x') = \mu \phi(x'')$, so equation~\eqref{intertwining diagram commutes} is satisfied, and thus $h_E \circ \phi_\infty(x) = \pi^*\circ h(x)$.
\end{proof}

\end{document}